\documentclass{amsart}
\usepackage{amsfonts,amssymb,amscd,amsmath,enumerate,verbatim,calc}
\usepackage[all]{xy}

\newcommand{\CM}{Cohen-Macaulay}

\newcommand{\n}{\mathfrak{n} }
\newcommand{\m}{\mathfrak{m} }
\newcommand{\M}{\mathfrak{M} }

\newcommand{\q}{\mathfrak{q} }
\newcommand{\p}{\mathfrak{p} }

\newcommand{\A}{\mathfrak{a} }
\newcommand{\F}{\mathcal{F} }
\newcommand{\G}{\mathcal{G} }
\newcommand{\Hc}{\mathcal{H} }
\newcommand{\Z}{\mathbb{Z} }

\newcommand{\Oo}{\mathcal{O} }
\newcommand{\rt}{\rightarrow}

\newcommand{\ov}{\overline}

\newcommand{\wh}{\widehat }

\newcommand{\image}{\operatorname{image}}

\newcommand{\depth}{\operatorname{depth}}

\newcommand{\charp}{\operatorname{char}}

\newcommand{\Tr}{\operatorname{Tr}}

\newcommand{\rank}{\operatorname{rank}}

\newcommand{\Ext}{\operatorname{Ext}}

\theoremstyle{plain}

\newtheorem{theorem}{Theorem}[section]
\newtheorem{corollary}[theorem]{Corollary}
\newtheorem{lemma}[theorem]{Lemma}
\newtheorem{proposition}[theorem]{Proposition}

\theoremstyle{definition}

\newtheorem{s}[theorem]{}
\newtheorem{remark}[theorem]{Remark}

\theoremstyle{remark}

\begin{document}

\title[invariant rings]{The Cohen-Macaulay property of invariant rings over ring of integers of a global field-II}
\author{Tony~J.~Puthenpurakal}
\date{\today}
\address{Department of Mathematics, IIT Bombay, Powai, Mumbai 400 076}

\email{tputhen@math.iitb.ac.in}


 \begin{abstract}
Let $A$ be the ring of integers of a number field $K$.  Let $G \subseteq GL_3(A)$ be a finite group. Let $G$ act linearly on $R = A[X,Y, Z]$ (fixing $A$) and let $S = R^G$ be the ring of invariants. Assume the Veronese subring $S^{<m>}$ of $S$ is standard graded. We prove that if for all primes $p$ dividing $|G|$, the Sylow $p$-subgroup of $G$ has exponent $p$ then for all $l \gg 0$ the Veronese subring $S^{<ml>}$  of $S$ is \CM. We prove a similar result if for all primes $p$ dividing $|G|$, the prime $p$ is unramified in $K$.
\end{abstract}
 \maketitle
\section{introduction}
Dear Reader, while reading this paper it is a good idea to have part 1 of this paper, see \cite{P}.
\s \label{intro} In this paper we assume $A$ is ring of integers in a global field, i.e., $A$ is one of the following two rings
\begin{enumerate}
  \item  the ring of integers of a number field $K$ (i.e., $K$ is a finite extension of $\mathbb{Q}$).
  \item the  ring of integers of finite extension of $F_q(t)$ (where $F_q$ is a finite field with $q$ elements).
\end{enumerate}
Let $G \subseteq GL_n(A)$ be a finite group. Let $R = A[X_1, \ldots, X_n]$ and let $G$ act linearly on $R$ (fixing $A$). In this paper we study \CM \ property
 of $S = R^G$ when $n = 3$. Previously the case when $n = 3$ was studied when $A = \Z$, see \cite{A}. In this paper an example of a group $G$ was given with   $\Z[X_1, X_2, X_3]^G$  \emph{not}   \CM. Note that $S$ need not be standard graded.  However a Veronese subring $S^{<m>}$ is standard graded (for filtration's see \cite[2.11]{Sch}; the general case is similar).
In this paper we show
\begin{theorem}\label{main}
  Let $A$ be the ring of integers of a number field $K$.  Let $G \subseteq GL_3(A)$ be a finite group. Let $G$ act linearly on $R = A[X,Y, Z]$ (fixing $A$) and let $S = R^G$ be the ring of invariants. Assume $S^{<m>}$ is standard graded. If for all primes $p$ dividing $|G|$, the Sylow $p$-subgroup of $G$ has exponent $p$ then for all $l \gg 0$ the Veronese subring $S^{<ml>}$  of $S$ is \CM. In particular $S^{<m>}$ is generalized \CM.
\end{theorem}
Let $p$ be a prime dividing $|G|$. Then the proof of Theorem \ref{main} reduces to the case when $A$ is the ring of integers of the Hilbert class field of $\mathbb{Q}(e^{2\pi i/p})$ and $G$ is a $p$-group.
We prove the following result
\begin{theorem}\label{g2}
Let $A$ be the ring of integers in the Hilbert class field of $\mathbb{Q}(e^{2\pi i/p})$ and let $P$ be a prime ideal of $A$ containing $p$. Set $(\Oo, (\pi)) = (A_P, PA_P)$. Let $G \subseteq GL_3(\Oo)$ be a $p$-group. Let $G$ act linearly on $R = \Oo[X,Y,Z]$. Assume that the natural map $G \rt GL_3(\Oo/(\pi))$ is trivial. Then $R^G$ is \CM.
\end{theorem}

The next result gives the \CM \ property of $R^G$ under a condition on $K$.
\begin{theorem}\label{unramify}
 Let $A$ be the ring of integers of a number field $K$.  Let $G \subseteq GL_3(A)$ be a finite group. Let $G$ act linearly on $R = A[X,Y, Z]$ (fixing $A$) and let $S = R^G$ be the ring of invariants.  Assume $S^{<m>}$ is standard graded. Assume for every prime $p$ dividing $|G|$, the prime $p$ is unramified in $K$. Then .for all $l \gg 0$ the Veronese subring $S^{<ml>}$  of $S$ is \CM. In particular $S^{<m>}$ is generalized \CM.
\end{theorem}

In part one of this paper we had conjectured that $A[X, Y]^G$ is always \CM. In this paper we prove
\begin{theorem}\label{d2}
Let $A$ be the ring of integers of a number field $K$.  Let $G \subseteq GL_2(A)$ be a finite group. Let $G$ act linearly on $R = A[X,Y]$ (fixing $A$) and let $S = R^G$ be the ring of invariants. Then for all $l \gg 0$ the Veronese subring $S^{<l>}$  of $S$ is \CM. In particular $S$ is generalized \CM.
\end{theorem}

\s We consider the Gorenstein property of ring of invariants. We show the following surprising  result:
\begin{theorem}\label{gor}
Let $A$ be ring of integers in a global field $K$. Let $G \subseteq GL_n(A)$ be a finite group. Set $R = A[X_1, \ldots, X_n]$ and $T = K[X_1, \ldots, X_n]$. Let $G$ act on $R$ ( resp. $T$) fixing $A$ (resp. $K$). Set $S = R^G$ and $U = T^G$. Assume $S$ is \CM.
Then the following assertions are equivalent.
\begin{enumerate}[\rm (i)]
  \item $S$ is Gorenstein.
  \item $U$ is Gorenstein.
\end{enumerate}
\end{theorem}

As an application of Theorem \ref{gor} we get:
\begin{corollary}\label{sln}
Let $A$ be the ring of integers of a number field $K$.  Let $G \subseteq SL_n(A)$ be a finite group. Let $G$ act linearly on $R = A[X_1, \ldots, X_n]$ (fixing $A$) and let $S = R^G$ be the ring of invariants. If $S$ is \CM \ then it is Gorenstein.
\end{corollary}
\s \label{rep}Next we discuss some representation theoretic questions which naturally arise in our study. Let $A$ be ring of integers in a global field and let $P$ be a maximal ideal of $A$. Set $(\Oo, (\pi)) = (A_P, PA_P)$. As $G$ is a subgroup of $GL_n(A)$, it is naturally a subgroup of $GL_n(\Oo)$. We assume that $\charp \Oo/\pi \Oo$ divides the order of $G$. Let $\widehat{\Oo}$ be the completion of $\Oo$ with respect to $(\pi)$. Then $G$ can be considered as a subgroup of $GL_n(\widehat{\Oo})$.  Thus $V = \widehat{\Oo}^n$ is a maximal \CM \ (= MCM)  $\widehat{\Oo}[G]$-module. We note that if $M$ is a MCM $\widehat{\Oo}[G]$-module then $M$ is free as an $\widehat{\Oo}$-module and so determines a representation (not necessarily injective) $G \rt GL_n(M)$ and so we may consider ring of invariants of $S(M) = S_{\widehat{\Oo}}(M)$ (the symmetric algebra of $M$ as an $\widehat{\Oo}$-module). Thus it is natural to consider representation theory of MCM $\widehat{\Oo}[G]$-modules
The simplest non-trivial case is when $G = \Z/2\Z$. We prove
\begin{proposition}
  [ with hypotheses as in \ref{rep}]\label{2-rep}
Let $G = \Z/2 \Z$. Then we have
\begin{enumerate}[\rm (1)]
  \item If $\widehat{\Oo}$ is of mixed characteristic then $\widehat{\Oo}[G] = \widehat{\Oo}[X]/(X^2 - 1)$ has finite representation type. Futhermore  if $M$ is an indecomposable MCM $\widehat{\Oo}[G]$-module then $S(M)^G$ is \CM.
  \item If $\widehat{\Oo}$ is equi-characteristic then $\widehat{\Oo}[G] = \widehat{\Oo}[X]/(X - 1)^2$ has infinite representation type. Nevertheless  if $M$ is an indecomposable MCM $\widehat{\Oo}[G]$-module then $S(M)^G$ is \CM.
\end{enumerate}
\end{proposition}
Let $T$ be a complete \CM \ local ring. A maximal \CM \ $T$-module $M$ is called rigid if $\Ext^1_T(M, M) = 0$. We show
\begin{proposition}
  [ with hypotheses as in \ref{rep}]\label{3-rep}
Let $G = \Z/3 \Z$. Assume that $\widehat{\Oo}$ is of mixed characteristic and that $\zeta$ a primitive third root of unity is in $\widehat{\Oo}$. Also assume that residue field of $\widehat{\Oo}$ has characteristic three.  Then we have $\widehat{\Oo}[G] = \widehat{\Oo}[X]/(X^3 - 1)$. If $M$ is an indecomposable rigid MCM $\widehat{\Oo}[G] $-module then $S(M)^G$ is \CM.
\end{proposition}

\s \label{ufd}Let $K$ be the quotient field of $A$. If $A$ is a UFD and if there does not exist a non-trivial homomorphism from $G \rt K^*$ then the usual proof, see \cite[1.5.7]{S}, shows that $S = R^G$ is also a UFD. However there are many cases where $A$ is NOT a UFD. Let $C(T)$ denote the class group of a normal domain $T$. As $R = A[X_1, \ldots, X_n]$ we get   $C(R) \cong C(A)$ is a finite group.
There is a standard homomorphism $i \colon C(S) \rt C(R)$, see
\cite[p.\ 489, Proposition 14]{B}. We show
\begin{theorem}
\label{class-inj}(with hypotheses as in \ref{ufd}). Assume there does not exist a non-trivial homomorphism from $G \rt K^*$. Then
the map $i$ is injective. Furthermore if $C(A) \neq 0$ then $C(S) \neq 0$.
\end{theorem}

We now describe in brief the contents of this paper.
In section two we discuss some preliminaries  that we need. In section three we give proof of local case of  Theorems \ref{main} and \ref{unramify}. In section four we give proofs of Theorems \ref{main} and \ref{unramify}. In section five we give proof of Theorem \ref{g2}.   In section six we give proof of Theorem \ref{d2}.
In section seven we give proof of Theorem \ref{gor} and Corollary \ref{sln}.
In section eight we consider a few representation theoretic considerations that
occur naturally in our work. In section nine we give proof of Theorem \ref{class-inj}. In section ten we indicate a fact of spectral sequences that we need.

\section{Preliminaries}
In this section we prove some preliminary results on local to global property of generalized  \CM \ property for invariant rings. We also study the relation between invariant rings of Sylow $p$-subgroups and rings of invariants of the group. Finally we study some change of rings problems.

\begin{lemma}\label{powers-cm}
Let $(\Oo, \m)$ be a local ring and let $R = \bigoplus_{n \geq 0}R_n$ be a graded Noetherian ring with $R_0 = \Oo$. Assume that $R^{<c>}$ is standard graded. Let $m$ be a positive integer and assume that $R^{<cml>}$ is \CM \ for all $l \gg 0$. Then $R^{<cl>}$ is \CM \ for all $l \gg 0$.
\end{lemma}
\begin{proof}
We may assume that $\Oo$ is complete. Let $\M$ be the graded maximal ideal of $R$. Then
$\M^{<l>}$ is the graded maximal ideal of $R^{<l>}$. Local cohomology behaves well with respect to the Veronese functor. In particular for all $i \geq 0$ we have
\[
\left(H^i_\M(R)\right)^{<l>} = H^i_{\M^{<l>}}(R^{<l>}).
\]
For $i \geq 0$ let $E_i = H^i_\M(R)^\vee$ be the Matlis dual of $H^i_\M(R)$. Note that
$E_i$ is a finitely generated $R$-module. We note that as $R^{<c>}$ is standard graded we have $(E_i)_{cml} = (E_i^{<c>})_{ml} = 0$ for all $i \gg 0$ if and only if $(E_i)_{cl} = 0$ for all $l \gg 0$. It follows that
$\left(H^i_\M(R)\right)^{<cml>} = 0 $ for $l \gg 0$ if and only if $\left(H^i_\M(R)\right)^{<cl>} = 0 $ for $l \gg 0$. The result follows.
\end{proof}

  The following result is essential in our proofs.
 \begin{theorem}\label{sylow}
 Let $(\Oo, (\pi))$ be a DVR with residue field of characteristic $p > 0$. Let $G$ be a finite subgroup of $GL_n(\Oo)$ acting linearly on $R = \Oo[X_1, \ldots, X_d]$ (fixing $\Oo$). Choose a Sylow $p$-subgroup $H_p$ of $G$. Then
 \begin{enumerate}[\rm (1)]
 \item
  If $(R^{H_p})^{<l>}$ is \CM\  then $(R^G)^{<l>}$ is \CM.
  \item
  Assume $(R^{H_p})^{<c>}$ and  $(R^G)^{<t>}$ are standard graded. If $(R^{H_p})^{<cl>}$ is \CM \ for $l \gg 0$ then $(R^G)^{<tl>}$ is \CM \ for $l \gg 0$.
 \end{enumerate}
\end{theorem}

 \s\label{transfer}
 Let $(\Oo, \pi)$ be a DVR such that $\Oo/(\pi)$ is a field of characteristic $p$. Let $G$ be a finite subgroup of $GL_n(\Oo)$ acting linearly on $R = \Oo[X_1, \ldots, X_n]$ (fixing $\Oo$). Let $H$ be a $p$-Sylow subgroup of $G$. Define
 \begin{align*}
  \psi^G_H &\colon R^H \rt R^G, \\
  r  &\mapsto \frac{1}{|G \colon H|}\sum_{gH \in G/H} gr.
 \end{align*}
See \cite[2.4]{S} which shows that the above definition  is independent of the choice of elements representing the cosets $G/H$. Furthermore note that $p$ does not divide $|G \colon H|$ and so it is a unit in $\Oo$. The following properties of $\psi^G_H$ are established in \cite[2.4]{S}:
\begin{enumerate}
 \item $\psi$ is $R^G$-linear.
 \item $\psi^G_H$ is a splitting of the inclusion $R^G \hookrightarrow R^H$.
\end{enumerate}

 \begin{proof}[Proof of Theorem \ref{sylow}]
 Set $H = H_p$.
  Let $\M_G$ be the $*$-maximal ideal of $R^G$. Notice as $R^H$ is a finite $R^G$-module we get $\sqrt{\M_G R_H} = \M_H$ the $*$-maximal ideal of $R^H$.  We have a sequence of $R^G$-linear maps
  $$ R^G   \hookrightarrow R^H \xrightarrow{\psi^G_H} R^G $$
  such that the composite is identity. As local-cohomology is $R^G$-linear we get
  a sequence of graded $R^G$-linear maps for all $i$:
  $$ H^i_{\M_G}(R^G)  \rightarrow H^i_{\M_G}(R^H) \xrightarrow{H^i(\psi^G_H)} H^i_{\M_G}(R^G) $$
  such that the composite is identity. It follows that
  if $H^i_{\M_H}(R^H)_n = 0$ then
  $H^i_{\M_G}(R^G)_n = 0$. Note $\dim R^G  = \dim R^H = d + 1$

  (1) If $(R^H))^{<l>}$ is \CM \ then $H^i_{\M_H}(R^H)_{nl} = 0$ for all $n \in \Z$ and for $i = 0, 1, \ldots, d$. By our argument $H^i_{\M_H}(R^G)_{nl} = 0$ for all $n \in \Z$ and for $i = 0, 1, \ldots, d$. It follows that $(R^G)^{<l>}$ is \CM.

  (2) We note that $(R^H)^{<ctl>}$ is \CM \ for $l \gg 0$. So by (1) we get $(R^{G})^{<ctl>}$ is \CM \ for $l \gg 0$. By \ref{powers-cm} it follows that $(R^{G})^{<tl>}$ is \CM \ for $l \gg 0$. The result follows.
 \end{proof}

Next we consider global version of Lemma \ref{powers-cm} in the case of our interest.
\begin{lemma}\label{powers-cm-global}
Let $A$ be ring of integers in a number field $K$  and let  \\ $R = A[X_1, \ldots, X_n]$.
Let $G$ be a finite subgroup of $GL_n(A)$ acting linearly on $R$ and let $S = R^G$.
Assume that $S^{<c>}$ is standard graded.
Let $m$ be a positive integer and assume that $S^{<cml>}$ is \CM \ for all $l \gg 0$. Then $S^{<cl>}$ is \CM \ for all $l \gg 0$.
\end{lemma}
\begin{proof}
  Let $P_1, \ldots, P_r$ be all the primes in $A$ such that $P\cap \Z = p\Z$ where $p$ divides $|G|$. Fix $P = P_i$ for some $i$. We note that $(S_P)^{<c>}$ is standard graded.

If $S^{<cml>}$ is \CM \ for all $l \gg 0$ then $(S_{P_i})^{<cml>}$ is \CM \ for all $l \gg 0$ for all $i$. By \ref{powers-cm} it follows that $(S_{P_i})^{<cl>}$ is \CM \ for all $l \gg 0$ for all $i$.  If $Q$ is a prime in $A$ such that $Q\cap \Z = qZ$ where $q$ does not divide $|G|$ then $S_Q$ is \CM. So $(S_Q)^{<l>}$ is \CM \ for all $l \geq 1$. It follows that for any prime $K$ in $A$
$(S_K)^{<cl>}$ is \CM \ for all $l \gg 0$. So  $S^{<cl>}$ is \CM \ for all $l \gg 0$.
\end{proof}

We need to deduce \CM \ property of invariant ring of a finite group by considering its Sylow $p$-subgroups.
We show
\begin{proposition}\label{cm-sylow}
  Let $A$ be ring of integers in a number field $K$  and let  \\ $R = A[X_1, \ldots, X_n]$.
Let $G$ be a finite subgroup of $GL_n(A)$ acting linearly on $R$ and let $S = R^G$. For each prime $p$ dividing $|G|$ fix a Sylow $p$-subroup $H_p$ of $G$.
Then
\begin{enumerate}[\rm (1)]
  \item Fix $l \geq 1$. If $(R^{H_{p}})^{<l>}$ is \CM \ for all primes $p$ dividing order of $G$ then $(R^G)^{<l>}$ is \CM.
  \item Assume $(R^G)^{<c>}$ is standard graded. Assume for each $H_p$ (with $p$ dividing $|G|$) we have $(R^{H_{p}})^{<d_p>}$ is standard graded and  $(R^{H_{p}})^{<d_pl>}$ is \CM \ for all $l \gg 0$. Then $(R^G)^{<cl>}$  \CM \ for all $l \gg 0$.
\end{enumerate}
\end{proposition}
\begin{proof}
 (1) Let $Q_1, \ldots Q_s$ be all the primes in $A$ such that it contains primes in  $\Z$  dividing $|G|$. If $(R^{H_{p}})^{<l>}$ is \CM  \ then $(R^{H_{p}}_{Q_i})^{<l>}$ is \CM. By \ref{sylow} it follows that $(R^{G}_{Q_i})^{<l>}$ is \CM \ for all $i$. It follows that $(R^{G})^{<l>}$ is \CM.

 (2) Set $t = c\prod_{p}d_p$. Then note that $(R^{H_{p}})^{<tl>}$ is \CM \ for all $l \gg 0$. By (1), $(R^G)^{<tl>}$  \CM \ for all $l \gg 0$. By \ref{powers-cm-global} it follows that $(R^G)^{<cl>}$ is \CM \ for $l \gg 0$.
\end{proof}
We need to change rings to prove our result. The following result is crucial.
\begin{theorem}
  \label{AB} Let $A, B$ be two ring of integers of number fields with $A \subseteq B$. Let $G \subseteq GL_n(A)$ be a finite group. Note  $G$ can be considered as a subgroup of $GL_n(B)$. Let
  $R = A[X_1, \ldots, X_n] \subseteq T = B[X_1, \ldots, X_n]$. Let $G$ act on $R$ and $T$ (fixing $A$ and $B$ respectively). Then
  \begin{enumerate}[\rm (1)]
    \item  Fix $l \geq 1$. Then $(R^{G})^{<l>}$ is \CM \ if and only if $(T^{G})^{<l>}$ is \CM.
    \item Assume $R^{<c>}$ and $T^{<d>}$ is standard graded. Then $(R^G)^{<cl>}$ is \\ \CM \ for $l \gg 0$ if and only if $(T^G)^{<dl>}$ is \CM \ for all $l \gg 0$.
  \end{enumerate}
\end{theorem}
\begin{proof}
  Let $\{ P_1, \ldots. P_s\}$ be \emph{all} the primes in $A$ lying over some prime in $\Z$ dividing the order of $G$. For a fixed $P_i$, let $Q_{ij}$ be the primes in $B$ lying over $P_i$  for $j =1, \ldots, r_i$. We note that $B_{P_i}$ is a free $A_{P_i}$ module. It follows that $((T^G)_{P_i})$ is a free $((R^G)_{P_i})$-module generated in degree zero. So  $((T^G)_{P_i})^{<l>}$ is a free $((R^G)_{P_i})^{<l>}$-module.
  We note that for a fixed $Q_{ij}$ the graded ring $((T^G)_{Q_{ij}})^{<l>}$ is a flat $((R^G)_{P_i})^{<l>}$ algebra with zero dimensional fiber.

  (1) Notice $(R^{G})^{<l>}$ is \CM \ if and only if $((R^G)_{P_i})^{<l>}$  is \CM \ for all $i$. We also have $(T^{G})^{<l>}$ is \CM \ if and only if $((T^G)_{Q_{ij}})^{<l>}$  is \CM \ for all $ij$. As the graded ring $((T^G)_{Q_{ij}})^{<l>}$ is a flat $((R^G)_{P_i})^{<l>}$ algebra with zero dimensional fiber we have that  $((R^G)_{P_i})^{<l>}$  is \CM \ for all $i$ if and only if   $((T^G)_{Q_{ij}})^{<l>}$  is \CM \ for all $ij$. The result follows.

  (2) We have by (1),  $(R^{G})^{<cdl>}$ is \CM \  for all $l \gg 0$ if and only if $(T^{G})^{<cdl>}$ is \CM \ for all $l \gg 0$.
  We also have by \ref{powers-cm-global},  $(R^{G})^{<cdl>}$ is \CM \  for all $l \gg 0$ if and only if $(R^{G})^{<cl>}$ is \CM \  for all $l \gg 0$.
  Similarly we have $(T^{G})^{<cdl>}$ is \CM \  for all $l \gg 0$ if and only if $(T^{G})^{<dl>}$ is \CM \  for all $l \gg 0$. The result follows.
\end{proof}

\section{Proof of local case of  Theorems \ref{main}, \ref{unramify}}
In this section we prove local case of  Theorem \ref{main} and Theorem \ref{unramify}. By \ref{cm-sylow} it suffices to prove the result for a Sylow $p$-subgroup of $G$. Thus we may assume $G$ is a $p$-group.
We first discuss the hypotheses on the DVR $\Oo$.

\s \label{setup-case2-O} $(\Oo,(\pi))$ is a DVR such that $\Oo_\pi$ has characteristic zero and $\Oo/(\pi)$ has characteristic $p >0$. We assume that $\Oo$ contains $p^{th}$ roots of unity.  We assume that if $\xi$ is any  primitive $p^{th}$-root of unity then $\xi-1 = \pi u_\xi$ where $u_\xi$ is a unit in $\Oo$.

\begin{remark}
  The hypotheses on $\Oo$ is satisfied in the cases that we essentially need. See \cite[7.3, 11.3]{P}.
\end{remark}
  We prove our result by induction on order of $|G|$. Although we are primarily interested in ring of invariants of a polynomial ring, the intermediate rings obtained during induction are no longer polynomial rings. So  we need to work in a more general class of rings to make the induction work.

\s \label{induct} Set $\F_\Oo$ to be class of graded Noetherian rings $R = \bigoplus_{n \geq 0} R_n$ such that
\begin{enumerate}
  \item $R_0 = \Oo$ and $\dim R = 4$.
  \item $R$ and $R/\pi R$ are domains
  \item $R_\pi$ is \CM \ and $H^3_{(R_\pi)_+}(R_\pi)_n = 0$ for $n \geq 0$.
  \item $H^2_{R_+}(R)_j = 0$ for $j \leq 0$ and for all $j \gg 0$
  \item   $H^3_{R_+}(R)_{0} = 0$ and there exists $m$ (depending on $R$) such that $H^3_{R_+}(R)_{jm}$ is free $\Oo$-module for all $j \ll 0$.
\end{enumerate}
We note that $R = \Oo[X, Y, Z] \in \F_\Oo$.
 \s\label{action}\emph{Group actions:} Let $R\in \F_\Oo$. Let $Aut^*(R)$ be the set of automorphism $\sigma \colon R \rt R$ such that
\begin{enumerate}
  \item $\sigma(R_n) \subseteq R_n$ for all $n \geq 0$. Furthermore $\sigma$ on $R_0 = \Oo$ is  the identity,
  \item The map $R_n \rt R_n$ induced by $\sigma$ is $\Oo$-linear.
\end{enumerate}
We consider only finite subroups of $Aut^*(R)$. We also note that if $\sigma \in Aut^*(R)$ then $\sigma$ induces an automorphism on $R/\pi R$.

One of our main results is
\begin{theorem}\label{main-body}
(with hypotheses as above). Let $R \in \F_\Oo$ and let $G$ be a finite $p$-subgroup of $Aut^*(R)$. Then $R^G \in \F_\Oo$.
\end{theorem}
We first show that the class of rings $\F_\Oo$ is indeed an appropriate class of rings for our purposes. We show
\begin{proposition}\label{acm}
Let $R \in \F_\Oo$ and assume that $R^{<r>}$ is standard graded. Then $R^{<lr>}$ is \CM \ for all $l \gg 0$.
\end{proposition}
\begin{proof}
It suffices to prove that $R^{<rml>}$ is \CM \ for $l \gg 0$, see \ref{powers-cm}. Let $\n =(\pi, R_+)$ be the $*$-maximal ideal of $R$. We have a long exact sequence in cohomology
$$ \cdots  \rt H^{i-1}_{(R_\pi)_+}(R_\pi) \rt H^i_\n(R) \rt H^i_{R_+}(R) \rt
 H^{i}_{(R_\pi)_+}(R_\pi) \rt \cdots. $$
 Clearly $\depth R \geq 2$. We note that as $R_\pi$ is \CM \ we have isomorphism
 $H^2_\n(R) \cong H^2_{R_+}(R)$. As Veronese functor commutes with local cohomology by our assumption $R^{<l>}$ will have depth three for all $l \gg 0$.
 We note that if $E$ is a free $\Oo$-module then the map $E \rt E_\pi$ is injective.
  As $H^2_{(R_\pi)_+}(R_\pi) = 0$ and as $H^3_{R_+}(R)_0 $ and
   $H^3_{R_+}(R)_{rml} $ is free for
   $l \ll 0$ it follows that $H^3_\n(R)_0 = 0$ and
    $H^3_{\n}(R)_{rml} = 0$ for $l \ll 0$. Furthermore  clearly $H^3_{\n}(R)_{rml} = 0$ for $l \gg 0$.
    As Veronese functor commutes with local cohomology it follows that $R^{<rmt>}$ is \CM \ for $l \gg 0$.
\end{proof}

\s If $S$ is a commutative ring and $G$ is a finite group then $S[G]$ denotes the group ring of $G$ with coefficients in $S$. If $E$ is an $S[G]$-module then we let $H^i(G, E)$ denote the $i^{th}$ group cohomology of $E$ with respect to $S[G]$.
We need the following:

\begin{lemma}
\label{h0g2}  Let $G$ be a cyclic group of order $p$. Let $R \in \F_\Oo$. Then
\begin{enumerate}[\rm (1)]
  \item $H^1(G, R)_0 = 0$.
  \item $H^0_{S_+}(H^1(G, R))_0 = 0.$
  \item $H^0_{S_+}(H^2(G,R))_0 = 0$.
\end{enumerate}
\end{lemma}
\begin{proof}
(1) We note that $R_0 = \Oo$ and $G$ acts trivially on $\Oo$. The result follows.

(2) This follows from (1).

(3) Let $p = \pi^e$. Note $ H^2(G,R)_0  = \Oo/p \Oo$. Suppose if possible $t = [\pi^r ] \in H^0_{S_+}(H^2(G,R))_0$ for some $r \leq e - 1$. Let $v \in S_+$ such that $v t = 0$ and  the image of $v$ is non-zero in $S/\pi S$. We have $\pi^r v = \Tr(\alpha)$ for some $\alpha \in R$.
Let $\deg  v  = c$. We have a map
$$ 0 \rt K \rt R_c \xrightarrow{\sigma - 1} R_c. $$
Note $K$ has trivial action by $G$.
We have an exact sequence
$$ 0 \rt K \rt R_c \rt I \rt 0, \quad \text{where} \ I = \image(\sigma -1).$$
This sequence splits as $\Oo$-modules. Let $\alpha = \alpha_K + \beta$ along this decomposition. We have an injection $I \hookrightarrow I_\pi$. We note that the action of $G$ on $I_\pi$ is diagonizable. Let $\{ y_1, \ldots, y_s \}$ be an $\Oo_\pi$-basis of $I_\pi$
with $\sigma(y_i) = \xi_i y_i$ where $\xi$ is a $p^{th}$-root of unity. Note $\xi_i \neq 1$.  As $\Tr(y_i) = 0$ it follows that $\Tr(\beta) = 0$.
So
$$ \pi^r v = \Tr(\alpha) = \Tr(\alpha_K) = p \alpha_K = \pi^e \alpha_K. $$
It follows that $v = \pi^{e - r}\alpha_K \in \pi S$. This contradicts our assumption on $v$.
\end{proof}

We will also need the following result.
\begin{lemma}\label{g2-trivial}
Let $W$ be a free $\Oo$-module. Suppose $W$ is also a $\Oo[G]$-module where $G$ is cyclic of order $p$. Suppose the action of $\sigma$ on $W/\pi W$ is trivial. Then $H^2(G, W) = W^G/pW^G$.
\end{lemma}
\begin{proof}
By \cite[9.6]{P},  there exists a $\Oo$-basis ${w_1, \cdots ,w_s}$ of $W$
such that $\sigma(w_i) = \xi_i w_i$ for all i, where $\xi_i$ is a $p^{th}$-root of unity.
Consider the part of sequence used to compute Group cohomology
$$ W \xrightarrow{\Tr} W \xrightarrow{\sigma - 1} W  \cdots$$
Let $w \in W$ be in the kernel of $W \xrightarrow{\sigma - 1} W $. Then $w \in W^G$.
Let $v \in W$. Say
$$v = \sum_{i = 1}^{s}\alpha_i w_i \ \quad \text{where $\alpha_i \in \Oo$}.$$
Notice $\Tr$ is $\Oo$-linear. Also note that if $\sigma(w_i) = \xi_i w_i$ where $\xi_i \neq 1$ then $\Tr(w_i) = 0$. So we have
$$ \Tr(v) = \sum_{i \mid \xi_i = 1} p \alpha_i w_i   \in p W^G.$$
Furthermore it is clear that $p W^G \subseteq \image \Tr$. Thus $\image \Tr = pW^G$.
The result follows.
\end{proof}

\s In our proof of Theorem \ref{main-body} we will need to use Ellingsrud-Skjelbred spectral sequences. A careful description  of these sequences is in part one of this paper, see \cite[section 5]{P}.
We now give
\begin{proof}[Proof of Theorem \ref{main-body}]
By a standard argument we may reduce to the case when $G$ is cyclic of order $p$. Set $S = R^G$.

(1) and (2) Clearly $S_0 = \Oo$ and $\dim S = 4$. We have $S$ is a domain. We have an exact sequence
$0 \rt R \xrightarrow{\pi} R \rt R/\pi R \rt 0$. Taking invariants we get that $S/\pi S$ is a subring of $(R/\pi R)^G$ which is a domain. So $S/\pi S$ is a domain.

(3) We note that $(R_\pi)^G = S_\pi$  is  \CM \ and
$$(H^3_{(R_\pi)_+}(R_\pi))^G = H^3_{(S_\pi)_+}(S_\pi).$$
It follows that $H^3_{(S_\pi)_+}(S_\pi)_n = 0$ for $n \geq 0$.

(4)  We use the Ellingsrud-Skjelbred spectral sequences with $\A = S_+$.
We first consider the spectral sequence
\[
(\beta_1)\colon \quad \quad \mathcal{E}_2^{p,q} = H^p(G, H^q_{S_+}(R)) \Longrightarrow H^{p+q}_{S_+}(G,R).
\]
We note that $  \mathcal{E}_2^{p,q} = 0$ for $q = 0, 1$ and when $q = 2$ we have
 $(\mathcal{E}_2^{p,2})_n = 0 $ for $n \leq 0$ and $n \gg 0$. So we have
 $H^{r}_{S_+}(G,R) = 0$ for $r = 0, 1$ and $H^{2}_{S_+}(G,R)_n = 0$  for $n \leq  0$ and $n \gg 0$.

 Next we consider the spectral sequence
 \[
(\alpha_1)\colon \quad \quad  E_2^{p,q} = H^p_{S_+}(H^q(G, R)) \Longrightarrow
H^{p+q}_{S_+}(G,R).
\]
 We have an exact sequence
 $$ 0 \rt H^0_{S_+}(H^1(G, R))  \rt H^2_{S_+}(S) \rt   E_3^{2,0} = E_\infty^{2,0}.$$
 By \ref{h0g2}(2),  $H^0_{S_+}(H^1(G, R))_ 0 = 0$. Furthermore  $H^0_{S_+}(H^1(G, R))_n = 0$ for $n \gg 0$ and for all $n < 0$. We note that $ E_\infty^{2,0}$ is a sub-quotient of $H^{2}_{S_+}(G,R)$. We have
$ H^{2}_{S_+}(G,R)_n = 0$  for $n \leq  0$ and $n \gg 0$.  It follows that
$H^2_{S_+}(S)_n  = 0$ for $n \leq 0$ and for $n \gg 0$.

(5) We have to consider two cases.

Case 1. The natural map $G \rt Aut(R/\pi R)$ is trivial. \\
We take a Veronese $T= R^{<mr>} $ standard graded such that it is \CM. We note that $\pi H^1(G, T) = 0$, see \cite[9.5]{P}. We have an exact sequence
$0 \rt H^1(G, T) \rt T/\pi T \rt H^2(G, T)$. By \ref{g2-trivial} $H^2(G, T) = T^G/p T^G$. It has positive depth.
Therefore  $\depth H^1(G, T)  \geq 2$.   We have an exact sequence
\[
0 \rt T^G/\pi T^G \rt T/\pi T \rt H^1(G, T) \rt 0.
\]
It follows that $\depth T^G/\pi T^G \geq 3$. So $T^G$ is \CM.   We have an exact sequence
 $$ 0 \rt 0=H^2_{T^G_+}(T^G/\pi T^G) \rt H^3_{T^G_+}(T^G) \xrightarrow{\pi} H^3_{T^G_+}(T^G). $$
 It follows that $H^3_{T^G_+}(T^G)_n$ is a free $\Oo$-module for $n = 0$ and for $n = mrl$ with $l \ll 0$.
 As Veronese functor behaves well with local cohomology we have $H^3_{S_+}(S)_{mrj}$ is free $\Oo$-module for $j = 0$ and for $j \ll 0$. As  $H^3_{(S_\pi)_+}(S_\pi)_0 = 0$ it follows that $H^3_{S_+}(S)_{0} = 0$ Thus $S \in \F_\Oo$.

 Case 2. The natural map $G \rt Aut(R/\pi R)$ is injective. \\
 We take a Veronese $T= R^{<mr>} $ such that it is standard and it is \CM. We may also assume $S^{<mr>}$ is standard graded and by (4) we may assume it has depth three. We also assume that $H^3_{T^G_+}(T)_n$ is a free $\Oo$-module for all $n \leq 0$ and $H^3_{T^G_+}(T)_0 = 0$.
We may also assume $H^0_{T^G_+}(H^1(G, T)) = 0$.

 By an argument similar to \cite[8.2, \ Claim 2]{P} it follows that $\dim H^2(G, T) \leq 2$. Also as  $p H^1(G, R) = 0$ it  follows that $\dim H^1(G, T) \leq 3$.

 Claim-1: $H^1_\n(H^1(G, T))_0 = 0$. \\
  We use the Ellingsrud-Skjelbred spectral sequence with $\A = \n$ the $*$-maximal ideal of $T^G$.
We first consider the spectral sequence
\[
(\beta_2)\colon \quad \quad \mathcal{E}_2^{p,q} = H^p(G, H^q_{\n}(T)) \Longrightarrow H^{p+q}_{\n}(G,T).
\]
We note that $  \mathcal{E}_2^{p,q} = 0$ for $q = 0, 1, 2$ and when $q = 3 ,4,$ we have
 $(\mathcal{E}_2^{p,q})_0 = 0 $. So we have
 $H^{r}_{\n}(G,T) = 0$ for $r = 0, 1, 2$ and $H^{j}_{\n}(G,T)_0 = 0$  for $j = 3, 4$

  Next we consider the spectral sequence
 \[
(\alpha_2)\colon \quad \quad  E_2^{p,q} = H^p_{\n}(H^q(G, T)) \Longrightarrow
H^{p+q}_{\n}(G,T).
\]
Note that we have a map
 \[
 E^{1,3}_2 = H^1_\n(H^3(G, T)) \rt E^{3,2}_2  = H^3_\n(H^2(G, T)) = 0
 \]
 Also note that $E_3^{4,1}$ is a sub-quotient of  $H^4_\n(H^1(G, T)) = 0$ (as $\dim H^1(G, T) \leq 3$. Also note that $\dim T^G = 4$, so $H^i_\n(-) = 0$ for $i \geq 5$.
 It follows that $E^{1,3}_2 = E^{1,3}_3 = E^{1,3}_\infty$. We note that
 $(E^{1,3}_\infty)_0$ is a sub-quotient of  $H^{4}_{\n}(G,T)_0 = 0$. So $(E^{1,3}_\infty)_0 = 0$. As $G$ is cyclic we have $H^1(G, T) = H^3(G, T)$. So we have $H^1_\n(H^1(G, T))_0 = 0$.

Claim:2   $H^{3}_{T^G_+}(T^G)_0 = 0$. \\
 We use the Ellingsrud-Skjelbred spectral sequence with $\A = T^G_+$.
We first consider the spectral sequence
\[
(\beta_3)\colon \quad \quad \mathcal{E}_2^{p,q} = H^p(G, H^q_{T^G_+}(T)) \Longrightarrow H^{p+q}_{T^G_+}(G,T).
\]
We note that $H^q_{T^G_+}(T) = 0$ for $q = 0, 1, 2$. Also note that $H^q_{T^G_+}(T) = 0$ for $q > 3$. So the spectral sequence collapses.
It follows that
 $H^{r}_{T^G_+}(G, T) = 0$ for $r = 0, 1, 2$ and $H^{3}_{T^G_+}(G, T) = H^0(G, H^3_{T^G_+}(T))$. In particular we have that $H^{3}_{T^G_+}(G, T)_0 = 0$.

  Next we consider the spectral sequence
 \[
(\alpha_3)\colon \quad \quad  E_2^{p,q} = H^p_{T^G_+}(H^q(G, T)) \Longrightarrow
H^{p+q}_{T^G_+}(G,T).
\]
We note that as $pH^i(G, R) = 0$ we get that $H^i_{\n }(H^i(G, R)) = H^i_{T^G_+}(H^i(G, R))$.  We have sequence
$$ (E_2^{1,1})_0 =  H^1_{T^G_+}(H^1(G, T))_0 \rt (E_2^{3,0})_0 = H^3_{T^G_+}(T^G)_0. $$
As $(E_2^{1,1})_0  = 0$ we get $(E_2^{3,0})_0 = (E_3^{3,0})_0.$
We note that $$(E_2^{0, 2})_0 = H^0_{T^G_+}(H^2(G, T))_0 = 0, \quad \text{see \ref{h0g2}(3).}$$ So $(E_3^{0, 2})_0 = 0$.
It follows that $$H^3_{T^G_+}(T^G)_0 =  (E_3^{3,0})_0 = (E_\infty^{3,0})_0.$$
So $H^3_{T^G_+}(T^G)_0 = 0$  as it is a sub-quotient of $H^{3}_{T^G_+}(G, T)_0  = 0$.

Claim:3 $H^1_{T^G_+}(H^1(G, R))_{jm}  = H^1_\n(H^1(G, R))_{jm} = 0$ for $j \ll 0$. \\
The first equality above holds since $p H^1(G, R) = 0$ implies $$H^1_{T^G_+}(H^1(G, R))  = H^1_\n(H^1(G, R)).$$
 Suppose the assertion in Claim-3 is not true. Then  after taking a further Veronese we may assume that $H^1_\n(H^1(G, T)) = \bigoplus_{n \leq -1}V_n$ with $V_{-1} \neq 0$.

By considering the spectral sequence $(\alpha_3)$ we have a map
$$ 0 \rt (E_3^{1,1})_{-1}  \rt  (E_2^{1,1})_{-1} = H^1_{T^G_+}(H^1(G, T))_{-1}  \rt H^3_{T^G_+}(T^G)_{-1}.$$
We note that $(E_3^{1,1})_{-1} = (E_\infty^{1,1})_{-1}$. The latter module is zero as it is a sub-quotient of $(H^2(G, T))_{-1} = 0$.  It follows that $V_{-1}$ is a sub-module of $H^3_{T^G_+}(T^G)_{-1}$ and it is a torsion $\Oo$-module.

After passing through a further Veronese we may choose $v \in T^{G}_1$ such that image of $v$ in $T^G/\pi T^G$ is non-zero and that it is
$H^1(G, T)$-regular. We have an exact sequence $0 \rt T(-1) \xrightarrow{v}  T \rt U \rt 0$. We note that $H^2_{T^G_+}(U)_n$ is free $\Oo$-module for $n \leq 0$ (as it is an
 sub-module of $H^3_{T^G_+}(T)_{n -1}$).  We note that $U_0 = \Oo$. It follows that
 $H^1(G, U)_0 = 0$.

 sub-claim 1: $H^2(U^G)_0$ is a free $\Oo$-module.

 We use the Ellingsrud-Skjelbred spectral sequences with $\A = U^G_+$.
We first consider the spectral sequence
\[
(\beta_4)\colon \quad \quad \mathcal{E}_2^{p,q} = H^p(G, H^q_{U^G_+}(U)) \Longrightarrow H^{p+q}_{U^G_+}(G,U).
\]
We note that $H^q_{U^G_+}(U) = 0$ for $q = 0, 1$. Also note that $H^q_{U^G_+}(U) = 0$ for $q \geq 3$. So the spectral sequence collapses.
It follows that
 $H^{r}_{T^G_+}(G, U) = 0$ for $r = 0, 1$ and $H^{2}_{U^G_+}(G, U) = H^0(G, H^2_{U^G_+}(U))$. In particular we have that $H^{2}_{U^G_+}(G, U)_0$ is a free $\Oo$-module.

  Next we consider the spectral sequence
 \[
(\alpha_4)\colon \quad \quad  E_2^{p,q} = H^p_{U^G_+}(H^q(G, U)) \Longrightarrow
H^{p+q}_{U^G_+}(G,U).
\]
We note that $H^0_{U^G_+}(H^1(G, U))_0 \subseteq H^1(G, U)_0 = 0$. It follows that
\[
H^2_{U^G_+}(U^G)_0 = (E_2^{2,0})_0 = (E_3^{2,0})_0 = (E_\infty^{2,0})_0
\]
Therefore $H^2_{U^G_+}(U^G)_0$ is a sub-module of $H^{2}_{U^G_+}(G, U)_0$ (see \ref{ltc}) which is a free $\Oo$-module. So $H^2_{U^G_+}(U^G)_0$ is a free $\Oo$-module.

We have an exact sequence $0 \rt T(-1) \xrightarrow{v}  T \rt U \rt 0$. Taking invariants and noting that $v$ is $H^1(G, T)$-regular we get an exact sequence
$$ 0 \rt T^G(-1)  \xrightarrow{v} T^G \rt U^G \rt 0. $$
By taking local cohomology with respect to $T^G_+$ we obtain an exact sequence
$$ 0\rt H^2_{T^G_+}(U^G)_0  \rt H^3_{T^G_+}(T^G)_{-1} \rt H^3_{T^G_+}(T^G)_{0} = 0. $$
By sub-claim-1 it follows that $ H^3_{T^G_+}(T^G)_{-1}$ is free $\Oo$-module. This forces $V_{-1} = 0$. Thus our assumption that $H^1_\n(H^1(G, R))_{jm} \neq 0$ for $j \ll 0$ is wrong. So  $H^1_\n(H^1(G, R))_{jm} = 0$ for $j \ll 0$.

By taking an appropriate Veronese we may assume that $H^1_\n(H^1(G, T)) = 0$. It follows that $H^1_{S_+}(H^1(G, T)) = 0$. By considering the spectral sequence $(\alpha_3)$ we get that $H^3_{T^G_+}(T^G)_n $ is free $\Oo$-module for $n \leq 0$. It follows that $S = R^G \in \F_\Oo$.
\end{proof}

\section{Proof of Theorems \ref{main} and \ref{unramify}}

We first give
\begin{proof}[Proof of Theorem \ref{main}]
For each prime $p$ dividing $|G|$ fix a Sylow $p$-subgroup $H_p$ of $G$. By \ref{cm-sylow} it suffices to prove that if $(R^{H_p})^{<d_p>} $ is standard graded then $(R^{H_p})^{<d_pl>} $ is \CM \ for all $l \gg 0$. Thus it suffices to assume $G$ is a $p$-group. By our assumption $G$ has exponent $p$. Let $K$ be the composite of the quotient field of $A$ and the Hilbert class field of $\mathbb{Q}(e^{2\pi i/p})$. Let $B$ be the ring of integers of the Hilbert class field of $\mathbb{Q}(e^{2\pi i/p})$ and let $C$ be the ring of integers of $K$. By \ref{AB} it suffices to prove the required assertion for $C$. We note that $G$ is conjugate to a group $H$ with $H$ a subgroup of $Gl_3(B)$. Again by \ref{AB} it suffices to prove the corresponding statement for $B$.
So we may assume that $G$ is a $p$-group and $A$ is the ring of integers of  the Hilbert class field of $\mathbb{Q}(e^{2\pi i/p})$.

 Assume $(A[X, Y, Z]^G)^{<c>} $ is standard graded.
Let $P$ be a prime ideal in $A$. Set $\Oo = A_P$. Let characteristic of the residue field of $\Oo$ be $q$. If $q \neq p$ then $\Oo[X, Y, Z]^G$ is \CM \ (and hence all its Veronese subrings). The essential case is when $q = p$. Then note that $\Oo$ satisfies the assumption in \ref{setup-case2-O}, see \cite[7.3]{P}. We note that $(\Oo[X, Y, Z]^G)^{<c>} $ is standard graded. So by Theorem \ref{main-body}, it follows that $(\Oo[X, Y, Z]^G)^{<cn>}$  is \CM \ for all $n \gg 0$. There are only finitely many primes $P$ in $A$ lying above $p$. It follows that $(A[X, Y, Z]^G)^{<cn>} $ is \CM \ for all $n \gg 0$.
\end{proof}

Next we give
\begin{proof}[Proof of Theorem \ref{unramify}]
For each prime $p$ dividing $|G|$ fix a Sylow $p$-subgroup $H_p$ of $G$. By \ref{cm-sylow} it suffices to prove that if $(R^{H_p})^{<d_p>} $ is standard graded then $(R^{H_p})^{<d_pl>} $ is \CM \ for all $l \gg 0$. Thus it suffices to assume $G$ is a $p$-group. By our assumption $p$ is unramified in $A$. Let $K$ be quotient field of $A$ and let $\zeta$ be primitive $p^{th}$-root of unity. Let
$B$ be the ring of integers in $K(\zeta)$. By \ref{AB} it suffices to prove the corresponding statement for $B$.

 Assume $(B[X, Y, Z]^G)^{<c>} $ is standard graded.
Let $P$ be a prime ideal in $B$. Set $\Oo = B_P$.
Let characteristic of the residue field of $\Oo$ be $q$. If $q \neq p$ then $\Oo[X, Y, Z]^G$ is \CM \ (and hence all its Veronese subrings). The essential case is when $q = p$. Then note that $\Oo$ satisfies the assumption in \ref{setup-case2-O}, see \cite[11.3]{P}.
We note that $(\Oo[X, Y, Z]^G)^{<c>} $ is standard graded. So by Theorem \ref{main-body},  it follows that $(\Oo[X, Y, Z]^G)^{<cn>}$  is \CM \ for all $n \gg 0$. There are only finitely many primes $P$ in $B$ lying above $p$. It follows that $(B[X, Y, Z]^G)^{<cn>} $ is \CM \ for all $n \gg 0$.

\end{proof}
\section{Proof of Theorem \ref{g2}}
In this section we give a proof of Theorem \ref{g2} by induction on order of $G$. The intermediate rings obtained in induction are no longer polynomial rings. So we need to work with a larger
class of rings to make the induction work. Let $A$ be ring of integers of Hilbert class field of $\mathbb{Q}(e^{2i\pi/p})$ and let $P$ be a prime ideal in $A$ containing $p$. Set $(\Oo, \pi) = (A_P, PA_P)$.

\s \label{thm2-class}  By $\Hc_{\Oo}$ we denote a class of rings $R$ with following properties.
\begin{enumerate}
  \item $R  = \bigoplus_{n \geq 0}R_n$ is a graded \CM \ ring with $R_0 = \Oo$.
  \item $\dim R = 4$.
  \item $R$ and $R/\pi R$ are domains.
\end{enumerate}
Note $\Oo[X,Y, Z] \in \Hc_{\Oo}$.
We give
\begin{proof}[Proof of Theorem \ref{g2}]
We note that if $H$ is a cyclic group of order $p$ which is normal in $G$ then note that the induced action of $G/H$ on $R^H/\pi R^H$ is also trivial. Thus it suffices to prove by induction on order of $G$  that if $G$ is cyclic of order $p$ and $R \in \Hc_{\Oo}$ then $S = R^G \in \Hc_\Oo$.  It is clear that $\dim S = 4$ and $S_0 = \Oo$. Also it is easy to prove that both $S$ and $S/\pi S$ are domains.

By \cite[9.5]{P}, we have $\pi H^1(G, R) = 0$. So we have an exact sequence $0 \rt H^1(G, R) \rt R/\pi R \rt H^2(G, R)$. By \ref{g2-trivial}, we have $H^2(G, R) = S/p S$ which has positive depth. So $\depth H^1(G, R) \geq 2$. The exact sequence $0 \rt S/\pi S \rt R/\pi R \rt H^1(G, R) \rt 0$ yields $\depth S/\pi S \geq 3$. So $\depth S \geq 4$. Thus $S$ is \CM.
\end{proof}

\section{Proof of Theorem \ref{d2}}
In this section we give a proof of Theorem \ref{d2}. We prove it by induction on order of $G$. The intermediate rings that are obtained are no longer polynomial rings. So we need to work on a more general class of rings to make the induction go through.

\s \label{dim3} Let $(\Oo, (\pi))$ be a  DVR of mixed characteristic. By $\G_{\Oo}$ we denote a class of rings $R$ with following properties.
\begin{enumerate}
  \item $R  = \bigoplus_{n \geq 0}R_n$ is a graded Noetherian  ring with $R_0 = \Oo$.
  \item $\dim R = 3$.
  \item $R$ and $R/\pi R$ are domains.
  \item $R_\pi$ is \CM.
  \item $H^i_{R_+}(R) = 0$ for $i = 0, 1$.
  \item $H^2_{R_+}(R)_0 = 0$  and $H^2_{R_+}(R)_n$ is free $\Oo$-module for $n \ll 0$.
\end{enumerate}
Note $\Oo[X, Y] \in \G_{\Oo}$.
We first prove the following:
\begin{proposition}\label{d2prop}
Let $R \in \G_\Oo$. Then for $l \gg 0$ the ring $R^{<l>}$ is \CM. In particular $R$ is generalized \CM.
\end{proposition}
\begin{proof}
Let $\n = (\pi, R_+)$ be the $*$-maximal ideal of $R$. Consider the long exact in cohomology
$$ \cdots \rt H^i_\n(R) \rt H^i_{R_+}(R) \rt H^i_{{R_\pi}_+}(R_\pi) \rt H^{i+1}_\n(R) \rt \cdots $$
As $H^i_{R_+}(R) $ and $H^i_{{R_\pi}_+}(R_\pi)$ is zero for $i = 0, 1$ we get $H^0_\n(R) = H^1_\n(R) = 0$ and as $H^2_{R_+}(R)_0 = 0$ we get   $H^2_\n(R)_0 = 0$. As $H^2_{R_+}(R)_n$ is free $\Oo$-module for $n \ll 0$ the map
$H^2_{R_+}(R)_n \rt H^i_{{R_\pi}_+}(R_\pi)_n$ is injective for $n \ll 0$. It follows that $H^2_\n(R)_n = 0$ for $n \ll 0$. Also clearly $H^2_\n(R)_n = 0$ for $n \gg 0$. As local cohomology behaves well with respect to the Veronese functor it follows that   $R^{<l>}$ is \CM \ for $l \gg 0$.
\end{proof}
Next we show
\begin{theorem}\label{d2local}
Let $p > 0$ be the characteristic of $\Oo/(\pi)$. Let $R \in \G_\Oo$.
  Let $G \subseteq Aut^*(R)$ be a $p$-group. Then $R^G \in  \G_\Oo$.
\end{theorem}
\begin{proof}
By a standard induction
we may assume that $G$ is cyclic of order $p$.
Clearly $S = R^G$ is a Noetherian graded domain of dimension three. Also as $S/\pi S $ is a subring of $(R/\pi R)^G$ which is a domain it follows that $S/\pi S$ is a domain.
As $R_\pi$ is \CM \ of characteristic zero it follows that $S_\pi = (R_\pi)^G$ is also \CM.

To prove the rest of properties  we use the Ellingsrud-Skjelbred spectral sequence with $\A = S_+$.
We first consider the spectral sequence
\[
(\beta_5)\colon \quad \quad \mathcal{E}_2^{p,q} = H^p(G, H^q_{S_+}(R)) \Longrightarrow H^{p+q}_{S_+}(G,R).
\]
We note that $  \mathcal{E}_2^{p,q} = 0$ for $q =  0, 1$ and when $q \geq 3$. So the spectral sequence collapses.
We have $H^{r}_{S_+}(G,R) = 0$ for $r = 0, 1$ and $H^{2}_{S_+}(G,R) = H^0(G, H^2_{S_+}(R))$. So we have $H^{2}_{S_+}(G,R)_0 = 0$ and $H^{2}_{S_+}(G,R)_n$ free $\Oo$-module for $n \ll 0$.

 Next we consider the spectral sequence
 \[
(\alpha_5)\colon \quad \quad  E_2^{p,q} = H^p_{S_+}(H^q(G, R)) \Longrightarrow
H^{p+q}_{S_+}(G,R).
\]
We note that for $p = 0, 1$ we have $E_2^{p, 0} = E_\infty^{p,0}$. The latter module is zero as $H^{r}_{S_+}(G,R) = 0$ for $r = 0, 1$. Thus $H^i_{S_+}(S) = 0$ for $i = 0, 1$.
We have an exact sequence
\[
H^0_{S_+}(H^1(G, R)) \rt H^2_{S_+}(S) \rt E_3^{2,0} \rt 0.
\]
We note that $E_3^{2, 0} = E_\infty^{2,0}$ which is a submodule of $H^{2}_{S_+}(G,R)$, see \ref{ltc}. It follows that $H^2_{S_+}(S)_n$ is free $\Oo$-module for $n \ll 0$. We note that as $R_0 = \Oo$ we get $H^1(G, R)_0 = 0$. So $H^0_{S_+}(H^1(G, R))_0 = 0$. Thus $H^2_{S_+}(S)_0 = 0$. So $S \in \G_\Oo$.
\end{proof}

We now give
\begin{proof}[Proof of Theorem \ref{d2}]
By \ref{cm-sylow} it suffices to assume $G$ is a $p$-group. Let $P$ be a non-zero prime  in $A$ and set $\Oo = A_P$. Let $q$ be the characteristic of the residue field of $\Oo$. If $q \neq p$ then $\Oo[X,Y]^G$ is \CM.  The essential case is when $q = p$.
We note that there are only finitely many primes in $A$ lying above $p \Z$.
 By Theorem
\ref{d2local} it follows that $(\Oo[X,Y]^G)^{<l>}$ is \CM \ for all $l \gg 0$. The result follows.
\end{proof}
\section{Proof of Theorem \ref{gor} and Corollary \ref{sln}}
We first give
\begin{proof}[Proof of Theorem \ref{gor}]
(i) $\implies$ (ii). This holds as $U$ is a localization of $S$.

(ii) $\implies$ (i). We have  to prove that $S_\M$ is Gorenstein for all maximal homogeneous ideal of $S$. For this it suffices to prove that if $P$ is a maximal ideal of $A$ then $S_P$ is Gorenstein. Set $(\Oo, (\pi))  = (A_P, PA_P)$. Set $V = S_P$. Note that $U = V_\pi$. Let $\n = (\pi, V_+)$ be the $*$-maximal ideal of $V$. As $V$ is \CM \ we have an exact sequence
\[
0 \rt H^n_{V_+}(V) \rt H^n_{(V_\pi)_+}(V_\pi) \rt H^{n+1}_\n(V) \rt 0.
\]
For all $i \in \Z$ set $F_i = H^n_{V_+}(V)_i$. By \cite[16.1.5]{BS},  $F_i$ is a finitely generated $\Oo$-module. As it is a submodule of $(F_i)_\pi$ it follows that $F_i$ is a free $\Oo$-module.
Set $F_i = \Oo^{a_i}$. Let $E$ be the injective hull of $\Oo/(\pi)$. We have an exact sequence $ 0 \rt \Oo \rt \Oo_\pi \rt E \rt 0$. It follows that
$$H^{n+1}_\n(V)_i = E^{a_i}$$.
As $V_\pi$ is Gorenstein the dual of $ H^n_{(V_\pi)_+}(V_\pi)$ is isomorphic to $V_\pi(a)$ for some $a \in \Z$. We have
\begin{equation*}
  \sum_{i\in \Z} \dim_{ V_\pi }  H^n_{(V_\pi)_+}(V_\pi)_{-i}z^i  = \sum_{ i \in \Z} \dim _{ V_\pi }(V_\pi)_{i + a}z^i  = \sum_{ i \in \Z} \rank _{ \Oo }V_{i + a}z^i. \tag{$\dagger$}
\end{equation*}

We have an exact sequence $0 \rt S \xrightarrow{\pi} S \rt \ov{S} \rt 0$. This yields an exact sequence
$$0 \rt H^n_\n(\ov{S})_i \rt H^{n+1}_\n(V)_i \xrightarrow{\pi}   H^{n+1}_\n(V)_i \rt 0. $$.
It follows that $H^n_\n(\ov{S})_i  = k^{a_i}$ where $k = \Oo/\pi$. The dual of $H^n_\n(\ov{S})$ is isomorphic to the canonical module,  $\omega_{\ov{S}}$, of $\ov{S}$. By ($\dagger$) it follows that its Hilbert series is equal to the Hilbert series of  $\ov{S}(a)$. As $\ov{S}$ is a domain it follows that $\omega_{\ov{S}} = \ov{S}(a)$, see \cite[4.4.5]{BH}. So $\ov{S}$ is Gorenstein. It follows that $S$ is Gorenstein.
\end{proof}
Next we give
\begin{proof}[Proof of Corollary \ref{sln}]
As  $G \subseteq SL_n(A) \subseteq SL_n(K)$,  it follows that \\ $U = K[X_1, \ldots, X_n]^G$ is Gorenstein, see  \cite[6.4.9]{BH}. By \ref{gor} it follows that $S$ is Gorenstein.
\end{proof}
\section{Some Representation theoretic considerations}
We first give
\begin{proof}[Proof of Proposition \ref{2-rep}]
(1) We note that $T = \wh{\Oo}[X]/(X^2 - 1)$ is local with maximal ideal $\m = (2, X -1) = (X-1, X +  1)$. We note that $(X-1)(X + 1) = 0$ in $T$. It follows from a result of Takahashi, \cite[3.2.5]{T},
that $T$ is of finite representation type with three indecomposables, say $M, N , T$ with  $\m = M \oplus N$. We note that $\rank_\Oo M = \rank_\Oo N = 1$. It follows that $S(M)^G$ and $S(N)^G$ are \CM. We have that the induced action of $G$ on $S(T)/\pi S(T)$ is injective. So by \cite[1.4]{P} we get that $S(T)^G$ is \CM.

(2) Set $T = \Oo[X]/(X-1)^2 = \Oo[Z]/(Z^2)$. By \cite{BGS} we have indecomposable modules $\Oo, T, M_n$ for $n \geq 1$. Clearly $S(\Oo)^G = S(\Oo)$ is \CM. We also have $\rank_\Oo T = \rank_\Oo M_n = 2$ for $n \geq 1$, see \cite[4.1]{BGS}. It follows from \cite[1.4]{P} that $S(T)^G, S(M_n)^G$ are \CM.
\end{proof}

\begin{proof}[Proof of Proposition \ref{3-rep}]
We have $(X^3 - 1)  = (X-1)(X- \zeta)(X - \zeta^2) = f_1f_2f_3$. Set $T = \Oo[X]/(X^3 - 1)$. By \cite[4.3]{DH}, rigid $T$-modules are $\Oo[X]/(f_i)$, $\Oo[X]/(f_if_j)$ (with $i \neq j$) and $T$.
We note that $ M_i = \Oo[X]/(f_i) $ has rank one as an $\Oo$-module. So $S(M_i)^G$ is \CM. We also have $N_{ij} = \Oo[X]/(f_if_j)$ with $i \neq j$ has rank two as an $\Oo$-module. So by \cite[1.4]{P} we get that $S(N_{ij})$ is \CM. We note that the action of $G$ on $T$ is the cyclic permutation action on $\{ 1, X, X^2 \}$. By a result in
 \cite[Theorem 23]{A} it follows that $S(T)^G$ is \CM.
\end{proof}
\section{proof of Theorem \ref{class-inj}}
In this section we prove Theorem \ref{class-inj}. We need to recall how the map $i \colon C(S) \rt C(R)$ is constructed.

\s Let $T$ be an integrally closed domain. By $X(T)$ we denote the free group on height one primes of $T$. Let $P(T)$ be the group of principal divisors. Then $C(T) = X(T)/P(T)$ is the class group of $T$.

\s\label{defn-i} Let $T $ be a subring of $U$ such that $U$ is a finitely generated $T$-module. Also assume $U, T$ are integrally closed domains.
For a height one prime $Q$ of $U$ define
\[
e(Q) = \ell(U_Q/(Q\cap T)U_Q).
\]
For a height one prime $\q$ in $T$ define
\[
i(\q) = \sum_{Q\cap T = \q}e_Q \ Q
\]
and then extend by linearity $i \colon X(T) \rt X(U)$.
It can be shown that for $a \in T$ we have $i(div_T(a)) = div_U(a)$. So we have a map
$i \colon C(T) \rt C(U)$.

We now give:
\begin{proof}[Proof of Theorem \ref{class-inj}]
We note that if $\sigma \in G$ and if $Q$ is a prime ideal in $R$ then $\sigma(Q)$ is a prime ideal in $R$. So $\sigma$ acts on $X(R)$. This action restricts to the natural action of $\sigma$ on $P(R)$. So $G$ acts on $C(R)$.
Let $\q$ be a prime ideal of height one in $S$.
Then
\[
i(\q) = \sum_{Q\cap S = \q}e_Q \ Q
\]
We note that if $Q, Q^\prime$ lie above $\q$ then for some $\sigma \in G$ we have $\sigma(QR_\q) = Q^\prime R_\q$, see \cite[Chapter VII, 2.1]{L}.
So we have
\[
e_Q = \ell(R_Q/\q R_Q) = \ell(R_{Q^\prime}/\q R_{Q^\prime} ) = e_{Q^\prime}.
\]
Call this common value as $\theta(\q)$.
So we have
\[
i(\q) = \theta(\q)\left(\sum_{Q\cap S = \q}Q\right).
\]

Suppose $D$ is a divisor such that $i([D]) = 0$. We may assume that $D$ is effective.
Say $D = \sum_{i =1}^{m}a_i\q_i$ with $a_i \geq 0$. So we have
$$i(D) =  \sum_{i =1}^{m}a_i \theta(\q_i)\left(\sum_{Q\cap S = \q_i}Q\right). $$
Note $i(D)$ is also effective. Thus $i(D) = div_R(x)$ for some $x \in R$. As $div_R(x)$ is $G$-invariant it follows that $(x) = (\sigma(x))$ for every $\sigma \in G$. So $\sigma(x) = r_\sigma x$ where $r_\sigma \in R$. Comparing degrees we get that $r_\sigma$ has degree zero. In particular $r_\sigma \in A$. Let $t$ be the order of $\sigma$. We have
$x = \sigma^t(x) = r_\sigma^t x$. So $r_\sigma^t = 1$. Thus $r_\sigma$ is root of unity in $A$.

Let $\tau \in G$. Then $(\tau \sigma)(x) = r_\tau r_\sigma x$. So we have a group homomorphism $h \colon G \rt  K^*$ where $h(\sigma) = r_\sigma$. By assumption $h$ is trivial. So $r_\sigma = 1$ for all $\sigma$. Thus it follows that $x \in S$.

Let $div_S(x) = \sum_{j = 1}^{l}b_j\p_j$. Then
$$div_R(x)  = \sum_{j =1}^{l}b_j \theta(\p_j)\left(\sum_{Q\cap S = \p_j}Q\right). $$
As $i(D) = div_R(x)$, comparing coefficients in $X(R)$ we get $D = div_S(x)$. So $D$ is a principal divisor. The result follows.

(2) The map $f \colon C(A) \rt C(R)$ defined by $f(\q) = \q R$ is an isomorphism.
Let $\p$ be a non-principal ideal in $A$. The ideal $\p R$ is a prime ideal in $R$ of height one and so $\p R \cap S$ is a prime ideal in $S$ of height one.

If $C(S) = \{ 0 \}$ then the prime ideal $\p R \cap S $ is principal. Say $\p R \cap S = (x)$. Then necessarily $x \in A$ and we also have $\p = x A$ is principal. This is a contradiction.
\end{proof}

\s \label{ex-class} We now give examples of groups which satisfy the hypotheses of our theorem.
\begin{enumerate}
\item
$G = [G, G]$.
\item
If $\charp K = p > 0$ then any $p$-group $G$ satisfies the hypotheses of our theorem.
\item
If $A \subseteq \mathbb{R}$ then any group with odd number of elements will satisfy the hypotheses of our theorem.
\end{enumerate}

\section{A remark regarding spectral sequences}
In this section we discuss a fact regarding spectral sequences which is crucial for us.

\s\label{ltc} Let $E^{p,q}_a \Rightarrow H^{p+q}$ be a first quadrant cohomological spectral sequence. Then for all $n \geq 0$ we have
$$ E_\infty^{n, 0} \subseteq H^n. $$
See \cite[5.2.6]{W} for this fact.

\section{Conflict of interest statement}
The author have no conflict of interest to declare that are relevant to this article.

\section{data availability statement}

\emph{My manuscript has no associated data.}

\end{document}